\documentclass[12pt,leqno]{article}
\usepackage{amssymb,amsmath,amsthm}
\usepackage[all]{xy}
\usepackage[alphabetic,lite]{amsrefs}
\numberwithin{equation}{section}

\usepackage[top=2cm,bottom=2cm,left=2cm,right=4cm,marginparsep=0.3cm,marginparwidth=3cm,includefoot]{geometry}
\usepackage{mathtools}
\usepackage{causal}
\usepackage{mathrsfs}

\usepackage[colorlinks=true, pdfstartview=FitV, linkcolor=blue, citecolor=blue, urlcolor=blue]{hyperref}

\begin{document}

\title
{Topologies and sheaves on causal manifolds}

\author{Pierre Schapira}
\maketitle

\begin{abstract}
A causal manifold  $(M,\gamma)$ is a manifold  $M$ endowed with a closed  convex  proper cone $\gamma$ in the tangent bundle $TM$ satisfying suitable properties. Let $\lambda=\gammc\subset T^*M$ be the antipodal of the polar cone of $\gamma$.
  An open set $U$ of $M$ is called $\gamma$-open if its Whitney normal cone contains  $\Int(\gamma)$. Similarly,  $U$ is  called $\lambda$-open if the micro-support of the constant sheaf on $U$ is contained in $\lambda$. 
We begin by   proving that the two notions coincide. 

Next, we prove that if $(M,\gamma)$ admits a ``future time function'' (in particular if it is globally hyperbolic) the   functor of  direct images establishes  an equivalence of triangulated categories
between  the derived category of sheaves on $M$ micro-supported by $\lambda$ and the derived category of sheaves on $\Mg$, the manifold $M$ endowed with the $\gamma$-topology. This generalizes a result of~\cite{KS90} which treated 
the case of a constant cone in a vector space. 
\end{abstract}

\tableofcontents

\section*{Introduction}
We choose a commutative unital ring $\cor$ of finite global dimension and we shall consider sheaves of $\cor$-modules (in the derived setting). Our main reference on this subject is the book~ \cite{KS90}. 

Consider first a real finite dimensional vector space $\BBV$ and a closed convex proper cone $\theta\subset\BBV$ with vertex at $0$. Set 
$\gamma=\BBV\times\theta$. 
The $\gamma$-topology was first defined in loc.\ cit.~\S~3.5  (with slightly different notations): an open subset $V$ of $\BBV$ is $\gamma$-open if $V=V+\theta$. 
Denote by $\BBV_\gamma$ the set $\BBV$ endowed with the $\gamma$-topology and denote by $\phig\cl \BBV\to \BBV_\gamma$ the continuous map associated with the identity. Denote by $\Derb_{\gammc}(\cor_\BBV)$ the bounded derived category of sheaves on $\BBV$ micro-supported by $\gammc$, the antipodal of the polar cone to $\gamma$. 
It is  proved in loc.\ cit.\ that the functor of direct image $\roim{\phig}$ induces an equivalence of this triangulated category with the derived category $\Derb(\cor_{\BBV_\gamma})$. 

The aim of this paper is to generalize this result to causal manifolds.

A causal manifold  $(M,\gamma)$ is  a manifold $M$ endowed with  a closed  convex proper cone $\gamma$  in the tangent bundle $TM$ with the property that the projection $\tau_M\cl TM\to M$ is surjective when restricted to $\Int(\gamma)$ and $\Int(\gamma)$ is the union of the $\Int(\gamma_x)$ for $x\in M$. One also naturally defines the notion of a morphism of causal manifolds. 

The  $\gamma$-topology  $\Mg$ on $M$ and the map $\phig\cl M\to \Mg$ were extended to causal manifolds in~\cite{JS16}. A  subset $A$ of $M$ is   
a $\gamma$-set if its Whitney normal cone contains $\Int\gamma$. A $\gamma$-open subset is an open subset of $M$ which is a  $\gamma$-set. 
Our first result asserts that $U$ is $\gamma$-open if and only if 
 $\SSi(\cor_U)$, the micro-support of the constant sheaf on $U$, is contained in  $\lambda=\gammc$, the antipodal of the polar cone to $\gamma$.
 
One can then define the $\lambda$-future $\futl{x_0}$ of $x_0\in M$ as the smallest closed subset of $M$ containing $x_0$ and whose micro-support is contained in $\lambda^a$. This set is also the closure of $\futg{x_0}$,  the smallest $\gamma$-set of $M$ containing $x_0$. We get two preorders
$\preceql$ and $\preceqg$ on $M$ and $x\preceqg y$ implies $x\preceql y$. 

Next we introduce the notion of a  future time function $\tim\cl M\to\R$, a variant of the classical notion of a time function. This is a causal submersive and surjective morphism with the property that, for any $x_0\in M$, it is proper on the sets  $\futl{x_0}$ 
and  $\opb{\tim}\tim(x_0)\cap \futl{x_0}=\{x_0\}$. Recall that
 globally hyperbolic manifolds admit time functions (see~\cites{Ge70, MS08, FS11} among others). Also recall that we have used time functions  in~\cite{JS16} to obtain global propagation results for sheaves whose micro-support satisfy suitable conditions.  As a by-product, we globally  solve in loc.\ cit.\ the Cauchy problem for hyperbolic $\shd$-modules with hyperfunctions solution.

Here, we use future time functions to  prove our main result: for a causal manifold $(M,\gamma)$, define  $\Derb_{\gammc}(\cor_M)$ similarly as above
where $M=\BBV$. Assume that $(M,\gamma)$ is endowed with a future time function $\tim$. Then  the   functor 
$\roim{\phig}\cl \Derb_{\gammc}(\cor_M)\to\Derb(\cor_{M_\gamma})$ is an equivalence of triangulated categories with quasi-inverse $\opb{\phig}$.

 Note that there is a vast body of literature on spacetimes and Lorentzian manifolds, particularly concerning the wave equation and its solutions.
Usually, such spacetimes are defined through a quadratic form of signature $(+,-,\cdots,-)$.
Here, following~\cite{JS16}, we prefer a weaker formulation
which only makes use of a cone in the tangent space.

\vspace{3ex}
\noindent
{\bf Acknowledgements} 

\spa 
We  warmly thank St{\'e}phane Guillermou, Yuichi Ike and Beno{\^i}t Jubin  for useful comments. 

\section{Notations and background}\label{sect:not}
Unless otherwise specified, a manifold means a real $C^\infty$-manifold and a morphism of manifolds $f\cl M\to N$ is a map of class $C^\infty$. 

\subsubsection*{Vector bundles}
Let $p\cl E\to M$ be a real (finite-dimensional) vector bundle over $M$.  We shall identify $M$ with the zero-section of $E$ and we denote by $E^*$ the dual vector bundle.
\begin{itemize}
\item 
As usual, one denotes by $a\cl E\to E$ the antipodal map, $(x;\xi)\mapsto(x;-\xi)$. For a subset $A\subset E$, one sets $A^a=a(A)$. 
\item
A subset $\gamma$ of $E$ is \define{conic} (or is a cone) if it is invariant by the action of $\R_{>0}$, that is, $\gamma_x\subset E_x$ is a cone (with vertex at $0$)  for each $x\in M$.
Here, $\gamma_x$ is the restriction of $\gamma$ to the fibre $E_x$.
If $\gamma$ is closed and conic, then its projection by $p$ on $M$ coincides with its intersection with the zero-section of $E$ and we identify the zero-section of $E$ with $M$.
\item
A cone $\gamma$ is proper if $\gamma_x$ contains no line for any $x\in M$.
\item
For a cone $\gamma\subset E$, one defines the polar cone $\gamma^\circ\subset E^*$ by 
\eq\label{eq:polarcone}
&&\gamma_x^\circ=\{\eta\in E_x^*;\langle \eta,\xi\rangle\geq0\mbox{  for all }\xi\in\gamma_x\}.
\eneq
\end{itemize}
One shall be aware that for a cone $\gamma\subset E$, its closure $\ol\gamma$ may be different from the union 
$\bigcup_{x\in M}\ol{\gamma_x}$ and similarly with its interior $\Int(\gamma)$.

\subsubsection*{Cotangent bundle}
Let $M$ be a real  manifold of class $C^\infty$.  We shall use the classical notations below.
\begin{itemize}
\item 
 $\tau_M\cl TM\to M$ is the  tangent bundle, $\pi_M\cl T^*M\to M$ the cotangent bundle,
 \item
 for $N$ a submanifold of $M$, $T_NM$ is the normal bundle, $T^*_NM$ the conormal bundle 
  (in particular,   $T^*_MM$ is the zero-section of $T^*M$), 
\item
 $\Delta_M$, or simply $\Delta$, is the diagonal of $M\times M$,
\item
for two manifolds $M$ and $N$ one denotes by $q_1$ and $q_2$ the first and second projections defined on $M\times N$,
\item
one denotes by $p_1$ and $p_2$ the first and second projections defined on $T^*(M\times N)\simeq T^*M\times T^*N$. One also uses the notation $p_2^a$ for the composition of $p_2$ and the antipodal map on $T^*N$. 
\end{itemize}

To a morphism $f\cl M\to N$ of manifolds, one associates the maps
\eq\label{eq:lagrcorresp}
&&\ba{l}\xymatrix{
TM\ar[rd]_-{\tau}\ar[r]^-{f'}&M\times_NTN\ar[d]^-{\tau}\ar[r]^-{f_\tau}&TN\ar[d]^-{\tau}\\
&M\ar[r]^-f&N,
}\quad\quad
\xymatrix{
T^*M\ar[rd]_-{\pi}&M\times_NT^*N\ar[d]^-{\pi}\ar[l]_-{f_d}\ar[r]^-{f_\pi}&T^*N\ar[d]^-{\pi}\\
&M\ar[r]^-f&N.
}\ea\eneq
We set
\eq\label{eq:Tf}
&&Tf\eqdot f_\tau\circ f'\cl TM\to TN,
\eneq
and call $Tf$ the tangent map to $f$.

One denotes by $\Gamma_f$ the graph of $f$ and  sets
\eq\label{eq:lambdaf1}
&&\Lambda_f\eqdot T^*_{\Gamma_f}(M\times N).
\eneq
Then there is a commutative diagram:
\eq\label{eq:lambdaf2}
&&\ba{l}\xymatrix{
&\Lambda_f\ar[d]^\sim\ar[ld]_-{p_1}\ar[rd]^-{p^a_2}&\\
T^*M&M\times_NT^*N\ar[l]^-{f_d}\ar[r]_-{f_\pi}&T^*N.
}\ea\eneq

\subsubsection*{Sheaves}
We shall mainly follow the notations of~\cite{KS90} and use some of the main results of this book. Here, $\cor$ denotes a commutative unitary ring with finite global dimension.
One denotes by $\Derb(\cor_M)$ the bouneded derived category of sheaves on $\cor$-modules on $M$ and calls an object of this category, a sheaf.
For a sheaf $F$, we shall in particular use its micro-support $\SSi(F)$, a closed $\R^+$-conic subset of $T^*M$.

\section{The linear case}\label{section:affin}

Let $\BBV$ be a finite dimensional real vector space and let $\theta$ be a closed convex proper cone with non-empty interior. We  set $\gamma=\BBV\times\theta$.

Let $V\subset\BBV$ be an open subset of $\BBV$. Recall, after~\cite{KS90}*{\S3.5} with slightly different notations, 
that $V$ is $\gamma$-open if $V=V+\theta$. The family of $\gamma$-open  subsets of $\BBV$  defines the  $\gamma$-topology $\BBV_\gamma$ on 
$\BBV$. One denotes by $\phig\cl\BBV\to \BBV_\gamma$ the continuous map associated with the identity of the set $\BBV$. 

\begin{definition}
Let $\BBV$ and $\gamma$ be as above, let $U$ be open in $\BBV$ and let $x_0\in\partial U$. We say that $U$ is 
$\gamma$-open in a neighborhood of $x_0$ if there exists a $\gamma$-open set $V$ and  an open neighborhood 
$U_0$ of $x_0$ such that $U\cap U_0=V\cap U_0$. 
\end{definition}

\begin{lemma}\label{le:affine1}
Let $\BBV$ and $\gamma$ be as above.
Let $U$ be open in $\BBV$ and let  $x_0\in\partial U$. Assume that for an open   neighborhood $U_0$ of $x_0$, 
 $U_0\times_\BBV\SSi(\cor_U)\subset U_0\times \theta^{\circ a}$. Then there exists an open subset $W$ of $\BBV$  which coincides with $U$ in a neighborhood of $x_0$ and such that $\SSi(\cor_W)\subset\BBV\times \theta^{\circ a}$.
\end{lemma}
\begin{proof}
We may assume that $\BBV$ is the Euclidian space $\R^n$ and $x_0=0$. We denote by 
$B_\epsilon(x)$ the open ball of center $x$ and radius $\epsilon>0$. We may assume that, setting     $\xi_0=(1,0,\dots,0)$,
$\xi_0\in\Int(\theta)$. Also set $y_\epsilon=(\epsilon,0,\dots,0)$ so that $x_0$ belongs to the boundary of 
$B_\epsilon(y_\epsilon)$. 
For $\eta\in\R$, set
\eqn
&&H_\eta=\{x\in\BBV;x_1\leq \eta\},\quad L_\eta=\{x\in\BBV;x_1=\eta\}.
\eneqn
Let $Z_{\epsilon,\eta}$ be the lens
\eqn
&&Z_{\epsilon,\eta}=B_{\epsilon+\eta}(y_\epsilon)\cap H_\eta,
\eneqn
and set 
\eqn
&&U_{\epsilon,\eta}=U\cap Z_{\epsilon,\eta}.
\eneqn
We fix $\epsilon>0$. For $0<\eta<\epsilon$, $x_0\in\Int Z_{\epsilon,\eta}$. Moreover, 
 for  $\eta$ small enough, $\SSi(\cor_{Z_{\epsilon,\eta}})\subset\BBV\times\theta^{\circ a}$.
 Since $U_0\times_\BBV\SSi(\cor_U)\subset U_0\times \theta^{\circ a}$, 
it follows (see~\cite{KS90}*{Prop.~5.4.14}) that for $\eta$ small enough, 
\eqn
&&\SSi(\cor_{U_{\epsilon,\eta}})\subset\BBV\times\theta^{\circ a},
\eneqn
and $U_{\epsilon,\eta}=U$ in a neighborhood of $x_0$. 

Now set
\eqn
&&W_{\epsilon,\eta}=\bigcup_{x\in L_\eta\cap U}x+\Int\, \theta.
\eneqn
Then $W_{\epsilon,\eta}$ is open, $\SSi(\cor_{W_{\epsilon,\eta}})\subset \BBV\times\theta^{\circ a}$ 
and $U_{\epsilon,\eta}\cup W_{\epsilon,\eta}$ satisfies the requested properties.
\end{proof}

\begin{lemma}\label{le:affine2}
Let $\BBV$ and $\gamma$ be as above  and  let $U$ be open in $\BBV$.
 Then $U$ is $\gamma$-open in $\BBV$ if and only if $\SSi(\cor_U)\subset \BBV\times \theta^{\circ a}$.
\end{lemma}
\begin{proof}
(i) Assume that $U$ is  $\gamma$-open.  Then 
$N(U)^\circ\subset \BBV\times\theta^\circ$ by~\cite{KS90}*{Prop.~5.3.7~(ii)} and thus $\SSi(\cor_U)\subset \BBV\times\theta^{\circ a}$ 
by~\cite{KS90}*{Prop.~5.3.8}.

\spa
(ii) Assume that   $\SSi(\cor_U)\subset \BBV\times\theta^{\circ a}$. 
It follows from~\cite{KS90}*{Prop.~5.2.3} that $\cor_U\simeq\opb{\phig}\roim{\phig}\cor_U$. 
By  Proposition 3.5.3~(i)  of loc.\ cit.,
for  $V$ open and convex,  we have the isomorphism
$\sect((V+\theta);\roim{\phig}\cor_U)\simeq\sect(V;\opb{\phig}\roim{\phig}\cor_U)$. Therefore, 
$\sect((V+\theta);\cor_U)\simeq\sect(V;\cor_U)$, hence 
$\sect((V+\theta);\cor_U)\neq0$ for any open set $V\subset U$. This last condition  implies $(V+\theta)\subset U$ hence $U=U+\theta$.  
\end{proof}

Let $\BBV$ and $\gamma$ be as above.
\begin{lemma}\label{le:affine3}
Let $U\subset \BBV$ be an open subset and let $x_0\in \partial U$. Then $U$ is $\gamma$-open in a neighborhood of $x_0$ if and only if there exists an open neighborhood $W$ of $x_0$ such that $W\times_\BBV\SSi(\cor_U)\subset W\times \theta^{\circ a}$.
\end{lemma}
\begin{proof}
Since both properties are local on the boundary of $U$ in $W$, we may assume by Lemma~\ref{le:affine1} that $W=\BBV$. 
Then apply Lemma~\ref{le:affine2}.
\end{proof}

\begin{definition}
Let $U$ be a convex  open subset of $\BBV$. One denotes by $\Ug$ the set $U$ endowed with the topology induced by $\BBV_\gamma$ and  by $\phiug\cl U\to \Ug$ the continuous map associated with the identity map of $U$. 
\end{definition}
 Note that an open subset $V$ of $U$ is open in $\Ug$ if and only if $V=(V+\theta)\cap U$. Indeed, if $W$ is open and $V=(W+\theta)\cap U$, then $V\subset (V+\theta)\cap U\subset (W+\theta)\cap U$.

One has a commutative diagram of topological spaces
\eq\label{diag:UVG}
&&\ba{l}\xymatrix{
\BBV\ar[r]^-{\phig}&\BBV_\gamma\\
U\ar[u]^-{i_U}\ar[r]_-{\phiug}&\Ug\ar[u]^-{i_\gamma}\, .
}\ea\eneq
One shall be aware that $U$ and $\BBV$ are locally compact Hausdorff  spaces contrarily to $\Ug$ and $\BBV_\gamma$ which are not Hausdorff and the classical results on proper direct images for sheaves  do not apply to these spaces. 

\begin{lemma}\label{le:opbphig}
Let  $U$ be  a convex open subset of $\BBV$. The functor  $\opb{\phiug}\cl \Derb(\cor_{\Ug})\to\Derb(\cor_U)$ takes its values in 
$\Derb_{U\times\theta^{\circ a}}(\cor_U)$.
\end{lemma}
\begin{proof}
Let $G\in\Derb(\cor_{\Ug})$. Consider Diagram~\eqref{diag:UVG}. Then 
\eqn
 \opb{\phiug}G&\simeq&  \opb{\phiug}\opb{i_\gamma}\roim{i_\gamma}G\simeq\opb{i_U}\opb{\phig}(\roim{i_\gamma}G).
 \eneqn
 Since $\SSi(\opb{\phig} (\roim{i_\gamma}G))\subset \BBV\times \gamma^{\circ a}$ by~\cite{KS90}*{Prop.~5.2.3}, we get the result.
 \end{proof}

\begin{proposition}\label{pro:eqv1}
Let  $U$ be  a convex open subset of $\BBV$.
The  functor $\roim{\phiug}\cl \Derb_{U\times\theta^{\circ a}}(\cor_U)\to\Derb(\cor_{\Ug})$ is an equivalence of triangulated categories with quasi-inverse $\opb{\phiug}$.
\end{proposition}
When $U=\BBV$,  this result is an immediate consequence of~\cite{KS90}*{Prop.~3.5.3, Prop.~5.2.3}. A careful reading of these propositions shows that the proofs
extend verbatim to the case where $\BBV$ is replaced with a convex open subset. 

Note that we shall not use Proposition~\ref{pro:eqv1} and that Theorem~\ref{th:main} will generalize it. 

\section{Topologies on causal manifolds}

\subsubsection*{Causal manifolds}

Consider first a real manifold $M$ (of class $C^\infty$) endowed with a closed convex proper cone $\gamma\subset TM$  containing the zero-section. 

\begin{definition}\label{def:causal}
\banum
\item
A causal manifold $(M,\gamma)$ is a manifold $M$ equipped with a closed convex proper cone  $\gamma\subset TM$ 
with the property   that the projection $\tau_M\cl\Int(\gamma)\to M$ is surjective  and $\Int(\gamma)=\bigcup_{x\in M}\Int(\gamma_x)$. In other words, 
 $(\Int(\gamma))_x=\Int(\gamma_x)$ and this set is non-empty for all $x\in M$. 
\item
To a causal manifold $(M,\gamma)$ one associates the cone $\lambda=\gammc\subset T^*M$. 
\item
A morphism of causal manifolds $f\cl(M,\gamma_M)\to(N,\gamma_N)$ is a morphism of manifolds 
$f\cl M\to N$ such that the map $Tf\cl TM\to TN$ satisfies  $Tf({\gamma_M})\subset{\gamma_N}$. 
Equivalently, a morphism of causal manifolds $f\cl(M,\lambda_M)\to(N,\lambda_N)$ is a morphism of manifolds 
$f\cl M\to N$ such that $f_d\opb{f_\pi}\lambda_N\subset\lambda_M$.  (The equivalence  is proved in~~\cite{JS16}*{Prop.~1.12}.)
\item
For $I$ an open interval of $\R$, we denote by $(I,+)$ the causal manifold $(I,\gamma)$ where  $\gamma=\{(t;v)\in TI; v\geq0\}$. 
\eanum
\end{definition}
One shall be aware that in~\cite{JS16}*{Def~1.13} the  cone $\gamma$ was open. 


Also note that the cone $\lambda$ satisfies
\eq\label{eq:good}
&&\lambda\cap\lambda^a= T^*_MM, \quad \lambda+\lambda=\lambda.
\eneq

\begin{definition}\label{def:ctcone}
(i) A constant cone contained in $\gamma$ is a triple $(\phi,U, \theta)$ where $U$ is open in $M$, $\phi \cl U \to \R^d$ is a chart, $\theta \subset \R^d$ is a closed  convex proper cone with non-empty interior   such that 
in this chart $U \times \theta\subset \gamma$ (that is, $\phi(U)\times \theta \subset T\phi(U\times_M\gamma)$).
A constant cone $(\phi,U, \theta)$ will often simply be denoted by $U \times \theta$.

\spa
(ii) A \define{basis of constant cones} for $\gamma$ is a family of constant cones  $\{U_i \times \theta_i\}_i$ contained in $\gamma$ such that  $\bigcup_iU_i=M$ and for any $x_0\in M$, $\Int(\gamma_{x_0})=\bigcup_{x_0\in U_i}\theta_i$.
\end{definition}

\begin{remark}\label{rem:constcone}
For a constant cone $(\phi, U,\theta)$, we shall often identify $U$ with its image by $\phi$ in $\R^d$.
 \end{remark}

\begin{lemma}
Let  $(M,\gamma)$ be a manifold $M$ equipped with a closed convex proper cone  $\gamma\subset TM$. Then 
$(M,\gamma)$ is a causal manifold if and only if there exists a basis of constant cones.
\end{lemma}
\begin{proof}
(i) Assume that there exists a basis of constant cones $\{U_i \times \theta_i\}_i$. Then  $\Int(\gamma)=\bigcup_iU\times\Int(\theta_i)$. Therefore, 
$(\Int(\gamma))_{x_0}=\bigcup_{x_0\in U_i}\theta_i  \neq\varnothing$ since $\theta_i$ has non-empty interior and $\bigcup_iU_i=M$. Moreover,
$\Int(\gamma_{x_0})=\bigcup_{x_0\in U_i}\theta_i$ by the definition of a basis of constant cones.

\spa
(ii) Conversely, assume that $\Int(\gamma_{x_0})\neq\varnothing$ for all $x_0\in M$ and $\Int(\gamma)=\bigcup_{x\in M}\Int(\gamma_x)$. We may assume that $M$ is open in $\R^d$. Then for any open convex cone $\theta$ with $\ol\theta\subset \Int(\gamma_{x_0})$, there exists an open neighborhood $U$ of $x_0$ such that $U\times\theta\subset\Int(\gamma)$. 
\end{proof}

\subsubsection*{$\gamma$-topology}

Let $(M,\gamma)$ be a causal manifold and recall that we have set $\lambda=\gammc$. 

For two  subsets $A,B\subset M$ and $x\in M$, the normal cone  $C_x(A,B)\subset T_xX$ is defined in~\cite{KS90}*{Def~4.1.1}. In a local coordinate system
\eqn
(x_0;v_0)\in C_{x_0}(A,B)&\Leftrightarrow& \mbox{there exists a sequence }\{(x_n,y_n,c_n)\}_n\subset A\times B\times\R^+\\
&&\mbox{ with }c_n(x_n-y_n)\to[n]v_0, x_n\to[n]x_0, y_n\to[n]x_0.
\eneqn
The strict  normal cone $N(A)\subset TM$ in defined in~\cite{KS90}*{Def~5.3.6} by
\eqn
&&N_x(A)=T_xX\setminus C_x(M\setminus A,A),\\
&&N(A)=\bigcup_{x\in M}N_x(A).
\eneqn
The cone $N(A)$ is open and convex. 

\begin{definition}[\cite{JS16}*{Def~1.18}]~\label{def:gammaset}
Let $(M,\gamma)$ be a causal manifold.
\banum
\item
A subset $A\subset M$ is a $\gamma$-set if  $\Int\gamma \subset N(A)$.
\item
A subset $U\subset M$ is  $\gamma$-open if it is both  an open subset and a $\gamma$-set. 
\eanum
\end{definition}
It follows from~\cite{KS90}*{(5.3.3)} that 
\eq\label{eq:strictNC3}
&&\mbox{$A$ is a $\gamma$-set}\Leftrightarrow\left\{\parbox{50ex}{there exists a basis of constant cones
  $\{U_i\times\theta_i\}_i$ for  $\gamma$ such that $U_i\cap(U_i\cap A+\Int\theta_i)\subset A$ for all $i$.
}\right.
\eneq
Remark that the property of being $\gamma$-open is a local property on the boundary $\partial U$.

One needs to check that this definition is in accordance with the previous definition of~\S~\ref{section:affin}.
\begin{lemma}\label{lem:affin0}
Assume that $M=\BBV$ is a finite dimensional vector space and $\gamma=\BBV\times\theta$ for a closed convex proper cone $\theta$ with non-empty interior. Then an open subset $U\subset\BBV$ is $\gamma$-open
 in the sense of Definition~\ref{def:gammaset} if and only if $U=U+\theta$. 
\end{lemma}
\begin{proof}
(i) Assume that $U=U+\theta$.  Let us choose an increasing family $\theta_n$ of closed convex proper cones such that $\bigcup_n\theta_n=\Int\,\theta$. Then 
$\{U\times\theta_n\}_n$ is a basis of constant cones and   $U$ is $\gamma$-open by~\eqref{eq:strictNC3}. 

\spa
(ii) Conversely, asssume that $U$ is $\gamma$-open  in the sense of Definition~\ref{def:gammaset}. This implies that  for any  $x_0\in\partial U$ and any closed convex cone $\lambda\subset \Int\theta$, there exists an open neighborhood $V$ of $x_0$ such that $V\cap(V\cap U+\lambda)\subset U$. 
Hence, $U=U+\lambda$ for any any closed convex cone $\lambda\subset \Int\,\theta$, which implies $U=U+\theta$. 
\end{proof}

\begin{proposition}[\cite{JS16}*{Prop~1.19, 1.21}]\label{pro:gammaset}
Let $(M,\gamma)$ be a causal manifold.
\bnum
\item
A set $A$ is a $\gamma$-set if and only if $\gamma_x\subset N_x(A)$ for every $x\in\partial A$. 
\item 
A set $A$ is a $\gamma$-set if and only if $M\setminus A$ is a $\gamma^a$-set.
\item
The family of $\gamma$-sets is closed under arbitrary unions and intersections.
\item
The family of $\gamma$-sets is closed under taking closure and interior.
\item
If  $A$ is a $\gamma$-set, then $\clos{\Int{A}} = \clos{A}$ and $\Int{\clos{A}} = \Int{A}$.
\item
If $A$ is a $\gamma$-set and $\Int{A} \subset B \subset \ol{A}$, then $B$ is a $\gamma$-set.
\enum
\end{proposition}

The $\gamma$-topology on $M$ was first defined in~\cite{KS90}*{\S~3.5}  when $M=\BBV$ is a vector space and 
$\gamma=M\times\theta$ for a closed convex cone $\theta\subset\BBV$. 

\begin{definition}[\cite{KS90}*{\S~3.5} and \cite{JS16}*{Def~1.23}]\label{def:gammatop}
Let $(M,\gamma)$ be a causal manifold.
\banum
\item
The $\gamma$-topology on $M$ is the topology for which the open sets are the $\gamma$-open sets. 
\item
One  denotes by $M_\gamma$ the space $M$ endowed with the  $\gamma$-topology  and by $\phi_\gamma\cl M\to M_\gamma$ the 
continuous map associated with the identity of the set $M$. 
\eanum
\end{definition}
One shall be aware that the closed sets for the $\gamma$-topology are not $\gamma$-sets in general. They are $\gamma^a$-sets. That is why it is better to avoid the terminology 
``a closed $\gamma$-set'' whose meaning could be a set which is  closed for the $\gamma$-topology as well as a closed set which is a $\gamma$-set. 
%

\subsubsection*{$\gamma$-topology and $\lambda$-topology}

\begin{definition}
Let $(M,\gamma)$ be a causal manifold  and let $\lambda=\gammc$.
\banum
\item
A locally closed  set $A\subset M$ is a $\lambda$-set if $\SSi(\cor_A)\subset \lambda$. 
\item
A $\lambda$-open set \lp resp.\ a  $\lambda$-closed set\rp\, 
is an open set \lp resp.\ a  closed set\rp\,which is also a $\lambda$-set.
\eanum
\end{definition}
Note that $U$ is $\lambda$-open if and only if $M\setminus U$ is $\lambda$-closed. 
\begin{remark} 
In~\cite{JS16}*{Not.~1.10 } we have set $\lambda=\gamma^{\circ}$, contrarily to the above notation.
\end{remark}

\begin{proposition}\label{pro:lambdaset}
Let $(M,\gamma)$ be a causal manifold.
\banum
\item
The family of $\lambda$-open sets is closed under arbitrary unions and finite intersections.
In particular, the family of $\lambda$-open subsets of $M$ defines a topology on $M$. 
\item
Similarly, the family of $\lambda$-closed sets is closed under arbitrary intersections and finite unions. 
\eanum
\end{proposition}
\begin{proof}
Since an open set $U$ is $\lambda$-open if and only if $M\setminus U$ is $\lambda$-closed, it is enough to prove (a).

\spa
(i) If $U_1$ and $U_2$ are $\lambda$-open, so is $U_1\cap U_2$ since $\cor_{U_1\cap U_2}\simeq\cor_{U_1}\tens\cor_{U_2}$ and 
$\SSi(\cor_{U_1}\tens\cor_{U_2})\subset\lambda$ by~\eqref{eq:good}.

\spa
(ii) If $U_1$ and $U_2$ are $\lambda$-open, so is $U_1\cup U_2$. This follows from (i) and the distinguished triangle
$\cor_{U_1\cap U_2}\to \cor_{U_1}\oplus\cor_{U_2}\to\cor_{U_1\cup U_2}\to[+1]$. 

\spa
(iii) Let $\{U_i\}_{i\in I}$ be a  family of $\lambda$-open subsets. Let us order $I$ by the relation $i\leq j$ if $U_i\subset U_j$. We may assume that $I$ is non empty and by (ii) we may assume that $(I,\leq)$ is directed. It then follows from~\cite{KS90}*{Exe.~5.7} that, setting 
$U=\bigcup_{i\in I}U_i$, $\SSi(\cor_U)\subset \lambda$.

\spa
(iv) Since $\SSi(\cor_M)= T^*_MM\subset\lambda$ and $\SSi(\cor_\varnothing)=\varnothing$, both $M$ and $\varnothing$ are $\lambda$-open and 
the proof is complete. 
\end{proof}

\begin{theorem}\label{th:gammasets}
Let $(M,\gamma)$ be a causal manifold and recall that we have set $\lambda=\gammc$. 
\banum
\item
An open set $U$ of $M$ is a $\gamma$-set if and only if it is a $\lambda$-set.
\item
A closed subset $Z$ of $M$ is a  $\gamma^a$-set if and only if it is a $\lambda$-set. In particular, a closed $\lambda$-set is closed for the $\gamma$-topology. 
\eanum
\end{theorem}
\begin{proof}
(a)--(i) We know by~\cite{KS90}*{Prop.~5.3.8} that $\SSi(\cor_U)\subset N(U)^{\circ a}$. If $U$ is a $\gamma$-set, then 
 $\Int\gamma\subset N(U)$ and thus $\SSi(\cor_U)\subset N(U)^{\circ a}\subset \gammc=\lambda$.

 \spa
 (a)--(ii) Now assume that $\SSi(\cor_U)\subset \gammc$. 
The problem is local on $M$ and we may assume that $M$ is open in a vector space $\BBV$. 
Let $U\subset M$ be an open subset and assume that $\SSi(\cor_U)\subset \lambda$. Let $W$ be an open and convex set and $\theta$ a open convex cone 
with $W\times_M\lambda\subset W\times\theta^{\circ a}$. Then $U\cap W$ is  $(U\cap W)\times\theta$-open  by Lemma~\ref{le:affine3}. 
It follows that $U$ is $\gamma$-open by~\eqref{eq:strictNC3}.

\spa
(b) Set $U=M\setminus Z$. Then $Z$ is a $\lambda$-set if and only if so is $U$. On the other-hand, 
$U$ is a $\gamma$-set if and only if $Z$ is a $\gamma^a$-set.
\end{proof}

\begin{definition}\label{def:lambdatop}
Let $(M,\gamma)$ be a causal manifold.
The $\lambda$-topology on $M$ is the topology for which the open sets are the $\lambda$-open sets. Hence, the closed sets are the closed $\lambda^a$-sets.
\end{definition}
It follows from Theorem~\ref{th:gammasets} that the $\gamma$-topology and the $\lambda$-topology coincide. 

\begin{proposition}\label{pro:new}
Let $f\cl (M,\gamma_M)\to(N,\gamma_N)$ be a morphism of causal manifolds and set 
$\lambda_M=\gamma_M^{\circ a}$, $\lambda_N=\gamma_N^{\circ a}$. Assume that 
$f$ is non-characteristic with respect to $\lambda_N$, that is:
\eq\label{eq:noncar1}
&&\opb{f_\pi}(\lambda_N)\cap\opb{f_d}(T^*_MM)\subset M\times_NT^*_NN.
\eneq
Then the map $f$ induces a continuous map $M_{\gamma_M}\to N_{\gamma_N}$.
\end{proposition}
\begin{proof}
It follows from the hypothesis and~\cite{KS90}*{Prop.~5.4.13} that the inverse image  functor $\opb{f}$ induces a functor
\eq\label{eq:noncar2}
&& \opb{f}\cl\Derb_{\lambda_N}(\cor_{N})\to \Derb_{\lambda_M}(\cor_{M}).
\eneq
Now let $V$ be a $\gamma_N$-open subset of $N$, that is, $V$ is open and $\SSi(\cor_V)\subset \lambda_N$.
Then $\SSi(\opb{f}\cor_V)\subset\lambda_M$ by \eqref{eq:noncar2} and the result follows since 
$\opb{f}\cor_V\simeq\cor_{\opb{f}V}$.
\end{proof}

\section{Preorders on causal manifolds}
 
\subsubsection*{Preorders}
Consider a preorder $\preceq$ on a manifold $M$ and its graph $\gPo\subset M\times M$.
Then
\eqn
&&\Delta\subset\gPo\, \quad \gPo\circ\gPo=\gPo\,.
\eneqn

For a subset $A\subset M$, one sets
\eq\label{eq:J+}
&&\left\{
\parbox{70ex}{
$\futo A = q_2(\opb{q_1}(A)\cap \gPo)=\{x\in M;\mbox{ there exists } y\in A\mbox{ with } y\preceq x\}$,\\
$\paso A = q_1(\opb{q_2}(A)\cap \gPo)=\{x\in M;\mbox{ there exists } y\in A\mbox{ with } x\preceq y\}$.
}\right.
\eneq
One calls  $\futo A$ the future of $A$ and $\paso A$ the past of $A$ (for the preorder $\preceq$). 
 For $x\in M$, we write $\futo x$ and $\paso x$ instead of $\futo{\{x\}}$ and $\paso{\{x\}}$ respectively. 
 Note that $\futo A=\bigcup_{x\in A}\futo x$ and  $\paso A=\bigcup_{x\in A}\paso x$. 

By its definition, one has
\eq\label{eq:deltag}
&&\gPo=\bigcup_{x\in M}\paso{x}\times\futo{x} = \bigcup_{x\in M}\{x\}\times\futo{x}.
\eneq
 
\begin{definition}\label{def:regcone}
Let $\preceq$ be a preorder on $M$.
\banum
\item
The preorder is closed if $\gPo$ is closed in $M\times M$.
\item
The preorder is proper if $q_{13}$ is proper on $\gPo\times_M\gPo$.
Equivalently, the preorder is proper if for any two compact subsets $A$ and $B$ of $M$, the so-called causal diamond
 $\futo A \cap \paso B$ is compact.
\eanum
\end{definition}

If the preorder $\preceq$ is closed, then  for any compact subset $A$ of $M$ the sets $\paso A$ and $\futo A$ are closed.
If the preorder is proper, then it is closed. A preorder is proper as soon as  $\futo x \cap \paso y$ is compact for any two $x,y\in M$.

\subsubsection*{The $\gamma$-preorder}

This subsection is, with some modifications, extracted from~\cite{JS16}*{\S~1.4}. 

Let $(M,\gamma)$ be a causal manifold. 
\begin{definition}\label{def:chronofut}
For $A\subset M$, we denote by $\futg{A}$ the intersection of all the $\gamma$-sets which contain $A$ and call it the $\gamma$-future  of $A$.
We set $\futg{x}=\futg{\{x\}}$.
\end{definition}
Note that a set $A$ is a $\gamma$-set if and only if $\futg{A}=A$.

Let us set
\eqn
&&x\preceqg y\mbox{ if and only if }y\in\futg{x}. 
\eneqn

\begin{lemma}\label{le:gDel3}
\banum
\item
The relation $x\preceqg y$ is a preorder.
\item
For $A\subset M$,  $\futg A=\bigcup_{x\in A}\futg x$. In other words, 
with the preceding notations, $\futg{A}=J^+_{\preceqg}(A)$.
  \eanum
\end{lemma}

\begin{proof}
(a) Let $y\in\futg{x}$ and $z\in\futg{y}$.
Then $\futg{x}$ is a $\gamma$-set which contains $y$ and $\futg{y}$ is the smallest $\gamma$-set which contains $y$.
Therefore, $\futg{y}\subset\futg{x}$ and $z\in\futg{x}$.

\spa
(b) The term on the right-hand side is a $\gamma$-set, contains $A$  and is contained in $\futg A$ (since $\futg x\subset\futg A$ for $x\in A$), whence the equality. 
It follows that $y\in\futg A$ if and only if there exists $x\in A$ such that $x\preceqg y$. 
\end{proof}

According to~\eqref{eq:J+}, one has 
\eq\label{eq:pasog}
&&\ba{l}
\futg{x}=\futog{x},\\
\pasog{y}=\pasog{\{y\}}\eqdot\{x;  x\preceqg y\}, \quad \pasog{A}=\bigcup_{y\in A}\pasog{y}.
\ea\eneq

\begin{remark}\label{rem:pasfutg}
Recall  the causal manifold $(I,+)$ of Definition~\ref{def:causal}~(d).
On $(I,+)$ the preorder $\preceqg$ is the usual order $ \leq$.
\end{remark}

\begin{proposition}\label{pro:Iopen}
Let $A\subset M$ be a closed subset of $M$.
Then $\futg{A} \setminus A$ is open.
\end{proposition}

\begin{proof}
One has $\Int(\futg{A}) \subset \Int(\futg{A}) \cup A \subset \futg{A}$.
Applying Proposition~\ref{pro:gammaset}~(iv), we get that $\Int(\futg{A}) \cup A$ is a $\gamma$-set.
Since it contains $A$, it contains $\futg{A}$.
Therefore $\Int(\futg{A}) \cup A=\futg{A}$  and $\futg{A} \setminus A = \Int(\futg{A}) \setminus A$ is open.
\end{proof}

\begin{corollary}\label{cor:gopen}
Let $U$ be open in $M$. Then $\futg{U}$ is $\gamma$-open. 
\end{corollary}
\begin{proof}
One has $\futg{U}=\bigcup_{x\in U}(\futg{x}\setminus \{x\})\cup U$.
\end{proof}

\begin{definition}\banum
\item
Let $I=[0,1]$. 
If a function $c \colon I \to M$ is left (resp.\ right) differentiable, we denote its left (resp.\ right) differential by $c'_l$ (resp.\ $c'_r$).
\item
A path $c \colon I \to M$ is a continuous piecewise smooth map.
\item
A path $c$ is causal if $c'_l(t), c'_r(t) \in \gamma_{c(t)}$ for any $t \in I$ and is  strictly causal if $c'_l(t), c'_r(t) \in (\Int\gamma)_{c(t)}$ for any $t \in I$.
\eanum
\end{definition}

Note that if $c_1$ and $c_2$ are two causal (resp.\ strictly causal) paths on $I$ with $c_1(1)=c_2(0)$, the concatenation $c=c_1\cup c_2$ (defined by glueing the two paths as usual) is causal (resp.\ strictly causal).

The next result is obvious (see also~\cite{JS16}*{Lem.~1.30}).
\begin{lemma}\label{le:jslocal2}
Let $(M,\gamma)$ be a causal manifold and consider a constant cone $U \times \theta$ contained in $\gamma$.
Then, for $y, z \in U$ with $z-y \in \theta$, we have $z\in\futg{y}$.
\end{lemma}

\begin{lemma}[see~\cite{JS16}*{Lem.~1.34}]\label{le:tbcle0}
Let $c\cl I\to M$ be a strictly causal path.
Then for $t_1\leq t_2$ with $ t_1,t_2\in I$ we have $c(t_1)\preceqg c(t_2)$.
\end{lemma}

\begin{proof}
It is enough to prove that for any $t_0\in I$, there exists $\alpha>0$ such that $c(t_0)\preceqg c(t)$ for $t\in(t_0,t_0+\alpha)$ and similarly $c(t)\preceqg c(t_0)$ for $t\in(t_0-\alpha,t_0)$.
We may assume $t_0=0$.
There exists a constant cone $U\times\theta$ contained in $\gamma$ and containing $(c(0),c'_r(0))$.
There exists $\alpha>0$ such that $c(t)-c(0)\in\theta$ for $t\in(0,\alpha)$.
By Lemma~\ref{le:jslocal2}, this implies $c(0)\preceqg c(t)$ for $t\in(0,\alpha)$.
The other case is similar, using $c'_l(0)$.
\end{proof}

\begin{proposition}[see~\cite{JS16}*{Lem.~1.35}]\label{pro:tbcle}
Let $A\subset M$.
One has $y\in\futg{A}$ if and only if $y\in A$ or there exists a strictly causal path $c\cl I\to M$ such that $c(0)\in A$, $c(1)=y$.
\end{proposition}

\begin{proof}
Let $B$ be the union of $A$ with the set of points that can be reached from $A$ by a strictly causal path.
We shall prove that $\futg{A}=B$.

\spa
(i) To prove that $B \supset \futg{A}$, it is enough to check that $B$ is a $\gamma$-set.
Choose a constant cone $U \times \theta$ contained in $\gamma$ with $U$ convex.
By~\eqref{eq:strictNC3}, it is enough to prove that $U\cap(B\cap U+\theta)\subset B$.
Let $y' \in B \cap U$ and $v' \in \theta$ with $y'+v' \in U$.
Since $U$ is convex, $c \colon I \to U, t \mapsto y'+tv'$ is a strictly causal path for $I$ a small enough neighborhood of $[0,1]$.
Since $y' \in B$, there exists a strictly causal path $\tilde{c}$ with $\tilde{c}(0)\in A$ and $\tilde{c}(1)=y'$.
Therefore, concatenating $\tilde{c}$ and $c$ proves that $y'+v' \in B$.

\spa
(ii) Let us prove that $B \subset \futg{A}$.
Let $y\in B, y\notin A$.
There exist $x\in A$ and a strictly causal curve $c$ going from $x$ to $y$.
Then $y\in \futg{x}$ by Lemma~\ref{le:tbcle0}.
Hence, $y\in\futg{A}$.
\end{proof}

\begin{corollary}[see~\cite{JS16}*{Lem.~1.41}]\label{cor:gDellambda}
One has $\pasog{y}=\futga{y}$.
\end{corollary}
\begin{proof}
This follows immediately from Proposition~\ref{pro:tbcle}. 
\end{proof}

\subsubsection*{The $\lambda$-preorder}

Let $(M,\gamma)$ be a causal manifold. Recall that $\lambda=\gammc$.
\begin{definition}\label{def:futl}
\banum
\item
For $A\subset M$, we denote by $\futl{A}$ the intersection of all closed $\lambda^a$-sets which contain $A$:
\eqn
&&\futl{A}=\bigcap Z\mbox{ with  $A\subset Z$ and $Z$ is a closed $\lambda^a$-set.}
\eneqn
\item
For $x\in M$, we set $\futl{x}=\futl{\{x\}}$. 
\eanum
\end{definition}
In other words,  $\futl{A}$ is the closure of $A$ for the $\lambda$-topology, hence, for the $\gamma$-topology. In particular, $\futl{A}=\futl{\ol A}$. 
Let us set
\eqn
&&x\preceql y\mbox{ if and only if }y\in\futl{x}. 
\eneqn

\begin{lemma}\label{le:gDel4}
The relation $x\preceql y$ is a preorder. 
\end{lemma}
\begin{proof}
If $y\in\futl{x}$, then $\futl{y}\subset\futl{x}$. Indeed, $\futl{x}$ is a closed $\lambda^a$-set which contains $y$. Hence, it contains 
 $\futl{y}$.
\end{proof}

\begin{notation}
We call $\preceqg$ the $\gamma$-preorder and $\preceql$ the $\lambda$-preorder.
\end{notation}
According to~\eqref{eq:J+}, one has 
\eq\label{eq:pasol}
&&\ba{l}
\futog{x}=\futg{x}\\
\pasol{y}=\pasol{\{y\}}\eqdot\{x;  x\preceql y\}, \quad \pasol{A}=\bigcup_{y\in A}\pasol{y}.
\ea\eneq

\begin{proposition}\label{pro:futUpasZ}
For $A\subset M$, one has $\futl A=\ol{\futg A}$. 
In particular, for $x,y\in M$,  $\futl{x}=\ol{\futg{x}}$ and $x\preceqg y$ implies  $x\preceql y$. 
\end{proposition}
\begin{proof}
(i) Let us prove the inclusion $\ol{\futg A}\subset \futl A$. 
The closed set $\futl A$ contains $A$ and  being a $\lambda^a$-set, it is a $\gamma$-set by Theorem~\ref{th:gammasets}. Therefore,
$\futg A\subset \futl A$ and  the result follows. 

\spa
(ii)  Let us prove the inclusion  $\futl A\subset \ol{\futg A}$. 
The closed set $\ol{\futg A}$ is a $\gamma$-set by Proposition~\ref{pro:gammaset}, hence a $\lambda^a$-set by Theorem~\ref{th:gammasets}. 
\end{proof}

\begin{corollary}
One has  $\ol{\bigcup_{x\in A}\futl x}= \futl A$.
\end{corollary}
\begin{proof}
(i) The inclusion $\futl{x}\subset\futl{A}$ implies the inclusion $\subset$ in the statement. 

\spa
(ii) The left-hand side is a $\gamma$-set thanks to Propositions~\ref{pro:gammaset} and~\ref{pro:futUpasZ}. It is thus a $\lambda^a$-set by Theorem~\ref{th:gammasets}. Therefore, it contains $\futl{A}$ by Definition~\ref{def:futl}.
\end{proof}
In the sequel, we set $\Delta_\gamma=\Delta_{\preceqg}$ and  $\Delta_\lambda=\Delta_{\preceql}$.

\begin{corollary}
One has  
\eqn
&&\gDel= I^+_{\gamma^a \times \gamma}(\Delta)\subset\lDel\subset  L^+_{\lambda^a \times \lambda}(\Delta)=\ol{\gDel}.
\eneqn
\end{corollary}
\begin{proof}
(i)  $I^+_{\gamma^a \times \gamma}(\Delta)\subset \gDel$. Since $\gDel$ contains the diagonal, il is enough to prove that it is a $(\gamma^a \times \gamma)$-set.
Each $\futga{x} \times \futg{y}$ is a $(\gamma^a \times \gamma)$-set by~\cite{JS16}*{Prop.~1.19~(iii)}. By using~\eqref{eq:deltag} we get the result by  Proposition~\ref{pro:gammaset}.

\spa
(ii)  $\gDel\subset I^+_{\gamma^a \times \gamma}(\Delta)$. Let $(x,y)\in \gDel$, that is $x\preceqg y$. By Proposition~\ref{pro:tbcle}, there exists a strictly causal path 
$c\cl I\to M$ with $c(0)=x$ and $c(1)=y$. Consider the path $\tw c=(c_1,c_2)\cl I\to M\times M$ given by $c_1(t)=c(1-t)$ and $c_2(t)=c(t)$. 
Then $\tw c$ is a strictly causal path in $(M\times M,\gamma^a\times\gamma)$. Since $\tw c(1/2,1/2)\in\Delta$, we get by 
Proposition~\ref{pro:tbcle} that $(x,y)\in   I^+_{\gamma^a \times \gamma}(\Delta)$. 

\spa
(iii) $\gDel\subset \lDel$ by Proposition~\ref{pro:futUpasZ}. 

\spa
(iv) $ L^+_{\lambda^a \times \lambda}(\Delta)=\ol{ I^+_{\gamma^a \times \gamma}(\Delta)}$ by Proposition~\ref{pro:futUpasZ}.
Hence, $ L^+_{\lambda^a \times \lambda}(\Delta)=\ol{\gDel}$ by (i)--(ii).

\spa
(v) $\lDel\subset  L^+_{\lambda^a \times \lambda}$. Indeed, 
 $\lDel=\bigcup_x(\{x\}\times\futl{x})$ by~\eqref{eq:deltag}. Since $\{x\}\times\futl{x}=\ol{\{x\}\times\futg{x}}$, we get  $\{x\}\times\futl{x}\subset\ol{\gDel}$.
\end{proof}

\begin{remark}
One has $y\in\futg{x}\Leftrightarrow x\in\futga{y}$
and  $y\in\futl{x}\Leftrightarrow y\in\ol{\futg{x}}$. However, 
the equivalence $y\in\ol{\futg{x}}\Leftrightarrow x\in\ol{ \futga{y}}$ is false, as shown by the following example well-known to specialists of the field.
\end{remark}

\begin{example}\label{exa:jub}
Let  $(x_1,x_2)$ be the coordinates on $\R^2$, $Z=\{(x_1,x_2);x_1\leq0, x_2=0\}$, $M=\R^2\setminus Z$.
 Let $\theta\subset \R^2$ be the closed cone $\theta=\{(x_1,x_2);x_2\geq \vert x_1\vert\}$ and let 
 $\gamma=M\times\theta\subset TM$. Now let $x=(-1,-1)\in M$, $y=(1,1)\in M$. Then 
 \eqn
 &&x\in\futla{y}\mbox{ but }y\notin\futl{x}.
 \eneqn
 To check this point, use Propositions~\ref{pro:futUpasZ} and~\ref{pro:tbcle}.
\end{example}

\subsubsection*{Future time functions}
Definition~\ref{def:Gcausal} is a variation on~\cite{JS16}*{Def.~1.50}. 
The terminology G-causal below is inspired by the name of Geroch.

\begin{definition}\label{def:Gcausal}
\banum
\item
A future time  function on a causal manifold $(M,\gamma)$ is a surjective and submersive causal morphism $\tim \cl (M,\gamma) \to (\R,+)$ which is proper on the sets $\futl{K}$    for all compact subsets  $K$ of $M$.
\item
A  future time  function $\tim$ is strict if for any  $x_0\in M$,  setting $t_0=\tim(x_0)$, 
the family $\{\futg{U}\cap\opb{\tim}(t_0-a,t_0+a)\}_{U,a}$ is a neighborhood system of $x_0$ when $U$ ranges over 
a neighborhood system of $x_0$ and $a\in\R_{>0}$.  
\item
A  $G^+$-causal manifold $(M,\gamma,\tim)$ is a causal manifold endowed with a strict future time  function $\tim$.
\item
A morphism of $G^+$-causal manifolds is a morphism of the underlying causal manifolds. 
\eanum
\end{definition}

A time function is a function $\tim$ which is a future time function both for $\gamma$ and $\gamma^a$.
\begin{remark}\label{rem:loctime}
Let $(U,\theta)$ be a constant cone contained in $\gamma$ (recall Definition~\ref{def:ctcone}). Then one easily constructs a future time function $\tim$ with respect to the causal manifold $(U,\theta)$ and $\tim$ will be a future time function for 
$(U,\gamma\vert_U)$. Hence, a causal manifold is always locally a 
 $G^+$-causal manifold.
 \end{remark}

We do not recall here the definition of a globally hyperbolic  spacetime and its links with that of G-causal manifold.
It is known, after the works (among others) of~\cites{Ge70, MS08, FS11} that a globally hyperbolic manifold admits  time functions.

\begin{lemma}\label{lem:timorder}
Let  $(M,\gamma)$  be a causal manifold and $\tim$ a future time function. Then 
\banum
\item
the function $\tim$ is strictly causal,
\item
 $x\preceqg y$ implies $\tim(x)\leq\tim (y)$.
 \eanum
\end{lemma}
\begin{proof}
(a) follows from~\cite{JS16}*{Prop.~1.9} and (b) from loc.\ cit.\ Lemma~1.46.
\end{proof}

\section{An  equivalence}

\subsubsection*{Micro-support of inverse images}

Let $(M,\gamma)$ be a causal manifold  and consider a constant cone $(U,\theta)$ contained in $\gamma$. 
We denote by $\Ug$ the set $U$ endowed with the topology induced by $\Mg$. We write $U_\theta$ instead of 
$U_{U\times\theta}$ and $\phit$ instead of $\phi_{U\times\theta}$. 

We get the commutative diagram of topological spaces:
\eq\label{diag:ctcone}
&&\ba{l}
\xymatrix{
U\ar[r]^-{\phit}\ar[d]_-{i_U}&U_\theta\ar[r]^-\rho&\Ug\ar[d]_-{i_{\Ug}}\\
M\ar[rr]^-{\phig}&&M_\gamma.
}\ea\eneq
 Thanks to Lemma~\ref{le:opbphig}, it gives rise to the commutative diagram
\eq\label{diag:thetag}
&&\ba{l}\xymatrix{
\Derb(\cor_M)\ar[d]_-{\opb{i_U}}
&&&\Derb(\cor_{\Mg})\ar[d]^-{\opb{i_\Ug}}\ar[lll]^-{\opb{\phig}}\\
\Derb(\cor_U)&\ar[l]^-\iota
      \Derb_{U\times\theta^{\circa}}(\cor_U)
                &\Derb(\cor_{\Ut})\ar[l]^-{\opb{\phit}}&\Derb(\cor_{\Ug})\ar[l]^-{\opb{\rho}}
}\ea\eneq

\begin{proposition}\label{pro:main2}
 Let $(M,\gamma)$ be a causal manifold. 
The  functor $\opb{\phig}\cl\Derb(\cor_{M_\gamma})\to\Derb(\cor_M)$ takes its values in $\Derb_{\gammc}(\cor_M)$.
\end{proposition}

\begin{proof}
The problem is local on $M$ and we may assume that $M$ is open in a vector space.
 Consider a constant cone $(U,\theta)$ contained in $\gamma$ with $U$ convex so that we may 
  apply Lemma~\ref{le:opbphig}.
Then 
\eq\label{eq:gammarho}
&&\opb{i_U}\circ\opb{\phig}\simeq\iota\circ\opb{\phit}\circ \opb{\rho}\circ\opb{i_{\Ug}}
\eneq
It follows that for $G\in\Derb(\cor_{M_\gamma})$, $\SSi(\opb{\phig}G\vert_U)\subset U\times \theta^{\circa}$.

Let $(x_0;\xi_0)\in \SSi(G)$.  There exists a family of  constant cones $\{(U_i\times\theta_i)\}_i$ such that 
$\Int\gamma_{x_0}=\bigcup_i\theta_i$. Hence, $\gammc_{x_0}=\bigcap_i \theta_i^{\circ a}$ in $T^*_{x_0}M$. 
Since $(x_0;\xi_0)\in \{x_0\}\times\theta_i^{\circ a}$ for all $i$, we get  that $(x_0;\xi_0)\in \gammc_{x_0}$.
\end{proof}

 \subsubsection*{Full faithfulness of  direct image}
 
Recall that for  $x_0\in M$, the set $\futl{x_0}$ is $\lambda^a$-closed hence is closed for the $\gamma$-topology, 

 \begin{lemma}\label{lem:germoim1}
 Let $(M,\gamma,\tim)$ be  a $G^+$-causal manifold and let $F\in \Derb_{\gammc}(\cor_M)$.
 Let $x_0\in M$, $t_0=\tim(x_0)$. Let $a>0$ and let $U$ be an open relatively compact neighborhood of $x_0$ such that 
 $\tim(U)\subset(t_0-a,t_0+a)$. 
 Then $\rsect(\futg{U};F) \isoto \rsect(\futg{U}\cap \opb{\tim}(t_0-a,t_0+a);F)$.
 \end{lemma}
  \begin{proof}
 (i) Set $U^+=\futg{U}$,   $G=\roim{\tim}\rsect_{U^+}F$. Note that, by the hypotheses, the function $\tim$ is proper on 
 $\ol{\futg{U}}$. Moreover, it is  increasing for the order $\preceqg$ by Lemma~\ref{lem:timorder},
  hence $\tim(U^+)\subset[t_0-a,+\infty)$. 
 
\spa
 (ii) The set $U^+$  being a $\gamma$-open subset, we get  $\SSi(\rsect_{U^+}F)\subset \lambda$.  Since $\tim$ is proper on $\supp(\rsect_{U^+}F)$,   $\SSi(G)\subset\{(t;\tau);t\geq t_0-a,\tau\leq0\}$ and it follows that 
 $\rsect(\R;G)\simeq\rsect((t_0-a,+\infty);G)\isoto \rsect((t_0-a,t_0+a);G)$.
  
 \spa
 (iii) One has  $\rsect(\futg{U};F) \simeq\rsect(M;\rsect_{U^+}F)\simeq \rsect(\R;G)$ and 
$\rsect((t_0-a,t_0+a);G)\simeq  \rsect(\futg{U}\cap \opb{\tim}(t_0-a,t_0+a));F)$.
\end{proof}

 \begin{proposition}\label{pro:ff}
 Let $(M,\gamma,\tim)$ be  a $G^+$-causal manifold.
 Then the functor
 $\roim{\phig}\cl \Derb_{\gammc}(\cor_M) \to\Derb(\cor_{\Mg})$ is fully faithful.
\end{proposition}
\begin{proof}
It is enough to prove the isomorphisms for all $j\in\Z$ and  $x_0\in M$
\eqn
&&H^j(\opb{\phig}\roim{\phig}F)_{x_0}\isoto H^j(F_{x_0}).
\eneqn
The left-hand side is obtained by applying $H^j$ and  taking the colimit with respect to the open neighborhoods $U$ of $x_0$ of $\rsect(\futg{U};F)$. The right-hand side is similarly obtained when replacing $\rsect(\futg{U};F)$ with $\rsect(U;F)$. 

Since the family $\{\futg{U}\cap  \opb{\tim}(t_0-a,t_0+a)\}_{U,a}$ is a basis of open neighborhoods of $x_0$ by Definition~\ref{def:Gcausal}, the result follows from Lemma~\ref{lem:germoim1}.
\end{proof}    

\subsubsection*{Main theorem}

 \begin{proposition}\label{pro:basic}
Let $f\cl X\to Y$ be a continuous map. 
and let $\sht$ be a full triangulated subcategory of $\Derb(\cor_X)$. Assume:
\banum
\item
 $f$ is surjective,
\item
 the functor $\opb{f}\cl\Derb(\cor_Y)\to \Derb(\cor_X)$ takes its values in $\sht$,
 \item
 the functor $\roim{f}\cl\sht\to \Derb(\cor_Y)$ is fully faithful. 
 \eanum
 Then $\roim{f}\cl \sht\to \Derb(\cor_Y)$  is an equivalence.
\end{proposition}
\begin{proof}
It is enough to prove that $\opb{f}$ is fully faithful, that is 
$\roim{f}\circ\opb{f}\simeq\id$. 

Let $G\in \Derb(\cor_Y)$ and set $F=\opb{f}G$. Let $y\in Y$ and choose $x\in X$ with $f(x)=y$. By the hypothesis, 
$(\opb{f}\roim{f}F)_{x}\simeq F_{x}$, hence $(\roim{f}F)_{y}\simeq F_{x}$.
Therefore, $(\roim{f}\opb{f}G)_{y}\simeq (\opb{f}G)_{x}$. 
The result follows since $(\opb{f}G)_{x}\simeq G_{y}$. 
\end{proof}

Applying  Propositions~\ref{pro:main2},~\ref{pro:ff} and~\ref{pro:basic}, with $X=M$, $Y=\Mg$ and $\sht=\Derb_{\gammc}(\cor_M)$,  we get:

\begin{theorem}\label{th:main}
Let $(M,\gamma,\tim)$ be a $G^+$-causal manifold. 
Then the functor
$\roim{\phig}\cl\Derb_{\gammc}(\cor_M)\to\Derb(\cor_{M_\gamma})$ is an equivalence of categories with quasi-inverse $\opb{\phig}$.
\end{theorem}

\begin{example}\label{exa:circle}
Let $M=\BBS^1$ be the circle and denote by $x$ a coordinate (hence, $x=x+2\pi$) on $M$, $(x;\xi)$ a coordinates system on $T^*M$.
Consider the cone $\gamma=\{(x;w);w\geq0\}\subset TM$. Then $(M,\gamma)$ is a causal manifold, $\gammc=\{(x;\xi);\xi\leq0\}$. The only $\gamma$-open sets are $M$ and $\varnothing$. Therefore, the only sheaves on $M_{\gamma}$ are the  constant sheaves.  
 The functor $\opb{\phig}$ is not essentially surjective. Indeed, the sheaf $\cor_I$ belongs to $\Derb_{\gammc}(\cor_M)$ for $I$ the interval $(a,b]$, $\vert b-a\vert<2\pi$. 
 This example does not contradict Theorem~\ref{th:main} since $(M,\gamma)$ is not $G^+$-causal.
 \end{example}

The next useful result already appeared in~\cite{Gu19}*{Cor.~III.~4.3~(ii)}. For the reader's convenience we repeat its statement and give a slightly more developed proof. 
\begin{corollary}
Let $M$ be a manifold and let $\lambda$ be a closed convex proper cone in $T^*M$. Let $F\in\Derb(\cor_M)$. Then 
$\SSi(F)\subset\lambda$ if and only if $\SSi(H^j(F))\subset\lambda$ for all $j\in\Z$.
\end{corollary}

\begin{proof}
It is enough to check that if $\SSi(F)\subset\lambda$, then the same property holds for all $H^j(F)$'s, the other implication being well-known.

Let $x\in M$ and assume that $\SSi(F)\cap T^*_xM\subset\lambda_x$. We may assume that $M$ is open in some vector space $\BBV$. Choose 
 a closed convex proper cone  $\lambda'\subset\BBV^*$ such that $\lambda_x\subset\Int(\lambda')$. Then $U\times_M\lambda\subset U\times\lambda'$ for an open neighborhood $U$ of $x$. Setting 
 $\gamma=U\times\lambda^{\prime\circ a}$, we have $F\vert_U\simeq\opb{\phig}\roim{\phig}F\vert_U$ and therefore
 (we  write $F$ instead of $F\vert_U$ for short):
\eqn
&&H^j(F)\simeq \opb{\phig}H^j\roim{\phig}F.
\eneqn
By Proposition~\ref{pro:eqv1} (with other notations), we get 
$\SSi(H^j(F))\cap T^*_{x}M\subset\lambda'$ and thus 
$\SSi(H^j(F))\cap T^*_{x}M\subset\lambda_x$. 
\end{proof}
\providecommand{\bysame}{\stLeavevmode\hbox to3em{\hrulefill}\thinspace}

\vspace*{1cm}
\noindent
\parbox[t]{21em}
{\scriptsize{
Pierre Schapira\\
Sorbonne Universit{\'e}, CNRS IMJ-PRG\\
4 place Jussieu, 75252 Paris Cedex 05 France\\
e-mail: pierre.schapira@imj-prg.fr\\
http://webusers.imj-prg.fr/\textasciitilde pierre-schapira/
}}


\begin{thebibliography}{15}



\bib{FS11}{article}{
 author={Fathi, Albert},
 author={Siconolfi, Antonio},
 title={On smooth time functions},
 journal={Mathematical Proceedings, CUP},
 volume={11},
 pages={437--449},
 date={2011},
 eprint={hal-00660452}
}


\bib{Ge70}{article}{
 author={Geroch, Robert},
 title={Domain of dependence},
 journal={J. Mathematical Phys.},
 volume={11},
pages={437--449},
year={1970}
}

\bib{Gu19}{article}{ 
author={Guillermou, St\'ephane},
title={Sheaves and symplectic geometry of cotangent bundles},
journal={Ast{\'e}risque},
volume={440},
publisher={Soc.\ Math.\ France},
 year={2023},
eprint={arXiv:1905.07341}
 }
 

\bib{JS16}{article}{
author={Jubin, Beno\^it},
author={Schapira, Pierre},
title={Sheaves and D-modules on causal manifolds},
journal={Letters in Mathematical Physics},
volume={16},
date={2016},
pages={607-648}
}




\bib{KS90}{book}{
 author={Kashiwara, Masaki},
 author={Schapira, Pierre},
 title={Sheaves on manifolds},
 series={Grundlehren der Mathematischen Wissenschaften},
 volume={292},
 publisher={Springer-Verlag, Berlin},
 date={1990},
 pages={x+512},
}


\bib{MS08}{article}{
 author={Minguzzi, Ettore},
 author={S{\'a}nchez, Miguel},
 title={The causal hierarchy of spacetimes},
 conference={
 title={Recent developments in pseudo-Riemannian geometry},
 },
 book={
 series={ESI Lect. Math. Phys.},
 publisher={Eur. Math. Soc., Z\"urich},
 },
 pages={299--358},
 year={2008}
}


\end{thebibliography}
\end{document}